\documentclass[reqno]{amsart}

\usepackage{amssymb, calrsfs, graphics, graphicx, enumerate, enumitem, url, xcolor, hyperref, mathrsfs, listings, comment}
\usepackage{tabularx}
\usepackage[ruled, lined, linesnumbered, longend]{algorithm2e}
\usepackage[justification=centering]{caption}

\newtheorem{theorem}{Theorem}[section]
\newtheorem{lemma}[theorem]{Lemma}

\newtheorem{proposition}[theorem]{Proposition}

\newtheorem{corollary}[theorem]{Corollary}

\numberwithin{equation}{section}

\theoremstyle{remark}

\lstdefinestyle{CStyle}{
    basicstyle=\footnotesize,
    breakatwhitespace=false,         
    breaklines=true,                 
    captionpos=b,                    
    keepspaces=true,                 
    numbers=left,                    
    numbersep=5pt,                  
    showspaces=false,                
    showstringspaces=false,
    showtabs=false,                  
    tabsize=2,
    language=C
}
\title[Permutations of $S_n$]{A conjecture of Glasby, Praeger, and Unger on  permutations of $S_n$}

\author[C. Bellotti]{Chiara Bellotti}
\address{School of Science\\ The University of New South Wales Canberra, Australia}
\thanks{Both authors were supported by Australian Research Council Discovery Project DP240100186. In addition, C.B.\ was supported by an Australian Mathematical Society Lift-off Fellowship.}
\email{c.bellotti@unsw.edu.au}

\author[T.~S. Trudgian]{Timothy S. Trudgian}
\address{School of Science\\ The University of New South Wales Canberra, Australia}
\email{t.trudgian@unsw.edu.au}

\begin{document}
\maketitle
\begin{abstract}
We prove a conjecture of Glasby, Praeger, and Unger concerning the symmetric group $S_{n}$. Let $\pi_{n}$ denote the proportion of elements of $S_{n}$ that are pre-$p$-cycles for some prime $p\in[2, n-3]$. We prove that $\pi_{n} > 1/3$ for all $n\geq 8$.
\end{abstract}
\section{Introduction and statement of result}
Consider the symmetric group $S_{n}$. A particular permutation $g\in S_{n}$ has a \textit{$k$-cycle} if it has a single cycle of length $k$, where $k>1$. As noted in the introduction to \cite{GlasbyPraegerUnger}, it is of interest to locate $p$-cycles, where $p$ is a prime. To facilitate such a location one searches for \textit{pre-$p$-cycles}, namely, particular permutations that power to a $p$-cycle. 
Unger \cite{Unger} proved that almost all permutations are pre-$p$-cycles. This `almost all' was quantified in \cite{GlasbyPraegerUnger}, wherein the authors showed that the proportion of elements that are pre-$p$-cycles for at least one $p\in[2, n-3]$ is at least $1- c\log\log n/\log n$, where $n$ is sufficiently large. Call this proportion $\pi_{n}$. 
We note that this problem has also been investigated by K\"{o}nig and Shin--- see Remark 3.1 in \cite[p.\ 1255]{TedescoRe}. 

The authors of \cite{GlasbyPraegerUnger} also raise the question about the size $\pi_{n}$ for all $n$, not just for sufficiently large $n$. Indeed, in Remark 11 in \cite{GlasbyPraegerUnger} it is shown that $\pi_{n} \geq 1/19$ for all $n\geq 5$. This improves on a result from \cite{Seress}. It was conjectured\footnote{Note that the conjecture \cite[Remark\,11]{GlasbyPraegerUnger} cannot hold for $n=5,6,7$ because a computation of the exact values of $\pi_n$ gives:
\begin{equation*}
\pi_5=\frac{1}{4}=0.25, \quad
\pi_6=\frac{35}{144}=0.24\ldots, \quad
\pi_7=\frac{47}{144}=0.326\ldots. 
\end{equation*}} that $\pi_{n} > 1/3$ may hold for all $n\geq 5$.  

Our main result is a proof of this bound on $\pi_{n}$  in the following theorem.
\begin{theorem}\label{th1}
For every integer $n\ge 8$, we have
$
\pi_n> \frac{1}{3}.
$
\end{theorem}

In \S \ref{jam} we outline some preliminary results on primes that refine slightly the corresponding results in \cite{GlasbyPraegerUnger}, which may be of independent interest. In \S \ref{lid} we apply our results on primes to $S_{n}$. Finally, in \S \ref{jar} we prove Theorem \ref{th1}.

\section{Preliminary lemmas}\label{jam}
\begin{lemma}\label{lem:prime-square-sum-sharper}
Let $A,a,b$ be integers with $50\le A\le a\le b$. Then
\begin{equation}\label{chuck}
\sum_{a<p\le b}\frac1{p^2}
\le
\frac{1+\dfrac{3}{\log A}}{a\log a}
-
\frac{1+\dfrac{3}{\log A}-\dfrac{3}{2\log b}}{b\log b}.
\end{equation}
In particular, taking $A=a$, we obtain
\[
\sum_{a<p\le b}\frac1{p^2}
\le
\frac{1+\dfrac{3}{\log a}}{a\log a}
-
\frac{1+\dfrac{3}{\log a}-\dfrac{3}{2\log b}}{b\log b}.
\]
\end{lemma}
\begin{proof}
We follow the proof of Lemma~7 in \cite{GlasbyPraegerUnger}. By partial summation,
\[
\sum_{a<p\le b}\frac1{p^2}
=
\frac{\pi(b)}{b^2}-\frac{\pi(a)}{a^2}
-
\sum_{k=a+1}^{b}\frac{-2k+1}{k^2(k-1)^2}\pi(k-1).
\]
Using the estimates\footnote{Note that the upper bound is valid for all real $x>1$ and the lower bound for all real $x\geq 17$. If we only require integral $x$, then, as in \cite{GlasbyPraegerUnger}, we can take $x\geq 11$.} from (3.2) and (3.5) in \cite{RS}
\begin{equation}\label{cents}
\frac{x}{\log x} \le \pi(x)\le \frac{x}{\log x}\left(1+\frac{3}{2\log x}\right)\qquad (x\ge 11, \quad x\in \mathbb{Z}),
\end{equation}
we obtain
\begin{align*}
\sum_{a<p\le b}\frac1{p^2}
&\le
\frac{1}{b\log b}\left(1+\frac{3}{2\log b}\right)
-\frac{1}{a\log a} \\
&\qquad\qquad
+\sum_{k=a+1}^{b}\frac{2k-1}{k^2(k-1)\log(k-1)}
\left(1+\frac{3}{2\log(k-1)}\right).
\end{align*}
Since $A\le a\le k-1$, we have
\[
1+\frac{3}{2\log(k-1)}\le 1+\frac{3}{2\log A}.
\]
Also, since $a\le k-1$, we have
\[
\frac1{\log(k-1)}\le \frac1{\log a}.
\]
Hence the sum above is at most
\[
\frac{1+\dfrac{3}{2\log A}}{\log a}
\sum_{k=a+1}^{b}\frac{2k-1}{k^2(k-1)}.
\]
Now $k-\tfrac12<k$, so
\[
\frac{2k-1}{k^2(k-1)}
=
\frac{2(k-\tfrac12)}{k^2(k-1)}
<
\frac{2}{k(k-1)}.
\]
Therefore
\begin{align*}
\sum_{k=a+1}^{b}\frac{2k-1}{k^2(k-1)\log(k-1)}
\left(1+\frac{3}{2\log(k-1)}\right)
&\le
\frac{2+\dfrac{3}{\log A}}{\log a}
\sum_{k=a+1}^{b}\frac1{k(k-1)} \\
&=
\frac{2+\dfrac{3}{\log A}}{\log a}\left(\frac1a-\frac1b\right).
\end{align*}
Substituting this into the previous estimate gives
\begin{align*}
\sum_{a<p\le b}\frac1{p^2}
&\le
\frac{1}{b\log b}\left(1+\frac{3}{2\log b}\right)
-\frac{1}{a\log a}
+\frac{2+\dfrac{3}{\log A}}{\log a}\left(\frac1a-\frac1b\right) \\
&=
\frac{1+\dfrac{3}{\log A}}{a\log a}
+\frac{1+\dfrac{3}{2\log b}}{b\log b}
-\frac{2+\dfrac{3}{\log A}}{b\log a}.
\end{align*}
Finally, since $\log a\le \log b$, we have
\[
-\frac{2+\dfrac{3}{\log A}}{b\log a}
\le
-\frac{2+\dfrac{3}{\log A}}{b\log b},
\]
and so
\[
\sum_{a<p\le b}\frac1{p^2}
\le
\frac{1+\dfrac{3}{\log A}}{a\log a}
-
\frac{1+\dfrac{3}{\log A}-\dfrac{3}{2\log b}}{b\log b},
\]
as claimed.
\end{proof}

\begin{corollary}\label{cor:prime-square-sum-sharper-real}
Let $A,a,b$ be real numbers with $50\le A\le a\le b$. Then
\begin{equation*}
\sum_{a<p\le b}\frac1{p^2}
\le
\frac{1+\dfrac{3}{\log \lfloor A\rfloor}}{\lfloor a\rfloor\log \lfloor a\rfloor}
-
\frac{1+\dfrac{3}{\log \lfloor A\rfloor}-\dfrac{3}{2\log \lfloor b\rfloor}}{\lfloor b\rfloor\log \lfloor b\rfloor}.
\end{equation*}
In particular, if $A=a$, then
\[
\sum_{a<p\le b}\frac1{p^2}
\le
\frac{1+\dfrac{3}{\log \lfloor a\rfloor}}{\lfloor a\rfloor\log \lfloor a\rfloor}
-
\frac{1+\dfrac{3}{\log \lfloor a\rfloor}-\dfrac{3}{2\log \lfloor b\rfloor}}{\lfloor b\rfloor\log \lfloor b\rfloor}.
\]
\end{corollary}

\begin{proof}
The primes $p\in (a, b]$ coincide with those in the interval
$\lfloor a\rfloor<p\le \lfloor b\rfloor$, so the result follows by applying
Lemma~\ref{lem:prime-square-sum-sharper} with $\lfloor A\rfloor$, $\lfloor a\rfloor$
and $\lfloor b\rfloor$.
\end{proof}
Noting that $b\geq A,$ we can write
$$1 + 3/\log A - 3/(2\log b) \geq 1 + 3/(2\log A).$$
Hence we can write the right-hand side of (\ref{chuck}) as $\lambda/(a\log a) - \mu/(b\log b)$, where $\lambda$ and $\mu$ are functions solely of $A$.
In \cite{GlasbyPraegerUnger} the authors take $A= 12$, which leads to $\lambda =2.22$ and $\mu=1.61$, after rounding. While this result is sufficient for our purposes in \S\S \ref{lid} and \ref{jar}, we note that the corresponding result in \cite{GlasbyPraegerUnger} has been used in \cite{DudekDJ,Campbell,Jonathans}, and hence appears to be of independent interest. As such, we make some small improvements below.

Dusart \cite[p.\ 55]{Dusart} gives an upper bound that improves upon that in (\ref{cents}):
\begin{equation}\label{cents2}
\frac{x}{\log x}\left( 1 + \frac{1}{\log x}\right) \le \pi(x)\le \frac{x}{\log x}\left(1+\frac{c}{\log x}\right)\qquad (x\ge 599),
\end{equation}
where $c=1.2762$. In fact, the upper bound holds for all $x>1$ (and is sharp in this range).
Using this, with $599\leq A\leq a \leq b$, one arrives at
$$\sum_{a<p\leq b} \frac{1}{p^{2}} \leq \frac{1 + 2c/\log A}{a\log a} - \frac{1 + (c+1)/\log A}{b\log b}.$$
With $A=599$ we arrive at $\lambda = 1.3992$ and $\mu = 1.355$, provided that $599\leq a \leq b$. We can extend this result by keeping $b\geq 599$ and taking some $a_{0} < 599$. Therefore
\begin{equation}\label{berry}\sum_{a_{0}<p\leq b} \frac{1}{p^{2}} = \sum_{a_{0}<p\leq 599} \frac{1}{p^{2}} + \sum_{599<p\leq b} \frac{1}{p^{2}} \leq \sum_{a_{0}<p\leq 599} \frac{1}{p^{2}} + \frac{1.3991}{599 \log 599}.
\end{equation}
We now aim to find the smallest constant $\lambda_{1}$ such that the right-side of (\ref{berry}) is bounded above by $\lambda_{1}/(a_{0} \log a_{0})$. This is a quick computation: we  may take $\lambda_{1} = 1.408$. This establishes that $\lambda = 1.408$ and $\mu = 1.355$, provided that $a\leq 599\leq b$. Since we already had this result when $599\leq a \leq b$ we need only extend to $a\leq b \leq 599$. This is another easy computation to perform. We therefore arrive at
\begin{equation}\label{coasters}
\sum_{a<p\le b}\frac1{p^2}
\le
\frac{1.41}{a\log a}
-
\frac{1.35}{b\log b}, \quad (2\leq a\leq b).
\end{equation}
We can refine this further if we ignore the negative term on the right-side of (\ref{coasters}). Indeed, in applications such as \cite{DudekDJ,Campbell,Jonathans}, one often wants a sum simply over $p>a$.
To make this refinement we use
$$\pi(x)\le \frac{x}{\log x}\left(1+\frac{1}{\log x} + \frac{2.51}{\log^{2} x}\right), \quad x\geq 355991,$$
from Dusart \cite[p.~55]{Dusart} to show that we may take $c= 1.1964$ in (\ref{cents2}, for all $x\geq 355991$. The same method of proof as above then gives that 
\begin{equation}\label{platters}
\sum_{p>a} \frac{1}{p^{2}} < \frac{1.19}{a\log a},
\end{equation}
provided that $a\geq 355991$.

Even if $a$ were arbitrarily large, one could not reduce the constant 1.19 to anything beneath 1. We ran a computation for all $a\leq 10^{8}$ and found that the `worst' constant on the right-side of (\ref{platters}) was $0.9557\ldots$ coming from $a=95364006$. This shows that (\ref{platters}) in fact holds for all $a\geq 2$. Further estimations, using more refined estimates on $\pi(x)$ are certainly possible, and may well show that a constant of $1$ is attainable on the right-side of (\ref{platters}).

\section{Applying prime bounds to $S_{n}$}\label{lid}

Fix $G =\mathrm{~S}_n$. Let $\lambda(g)=\left\langle 1^{m_1(g)} 2^{m_2(g)} \cdots n^{m_n(g)}\right\rangle$ be the partition of $n$ whose parts are the cycle lengths of $g$, and part $k$ has multiplicity $m_k(g)$. Define
$$
\begin{aligned}
\mathcal{P}_n & =\left\{p \mid a(n)<p \leqslant a(n)^{d(n)}\right\}, \\
T_n & =\left\{g \in G \mid m_p(g) \geqslant 1 \text { for some } p \in \mathcal{P}_n\right\}, \\
U_p(G) & =\left\{g \in G \mid m_p(g) \geqslant 1 \text { and } \sum_{k \geqslant 1} m_{k p}(g) \geqslant 2\right\}, \text { and } \\
U_n & =\bigcup_{p \in \mathcal{P}_n} U_p(G)
\end{aligned}
$$
\begin{proposition}\label{prop:remark11-large-range}
Let $a(n)$ and $d(n)$ be functions such that
\[
a(n)\ge 50,\qquad d(n)>1,\qquad a(n)^{d(n)}\le n-3
\]
for all $n$. Set
\[
b(n):=a(n)^{d(n)},\qquad
\mathcal P_n:=\{p\text{ prime}:a(n)<p\le b(n)\}.
\]
Let $\pi_n$ denote the proportion of elements of $S_n$ which are pre-$p$-cycles
for some prime $p$ with $2\le p\le n-3$. Then
\[
\pi_n\ge 1-\frac{2.015}{d}-(\log n+1-\log a)\frac{1+\dfrac{3}{\log 50}}{\lfloor a\rfloor\log \lfloor a\rfloor}-\frac{4.09\log n}{an\log a}.
\]
\end{proposition}

\begin{proof}
Write $a=a(n)$, $d=d(n)$ and $b=b(n)=a^d$. Let $T$, $U_p$ and $U$ be the sets
defined above.
Since $b\le n-3$, every element of $T_n\setminus U_n$ is a pre-$p$-cycle for some
prime $p\in\mathcal P_n$, and hence
\[
\pi_n\ge \frac{|T_n|-|U_n|}{|S_n|}.
\]
Set
\[
\mu_n:=\sum_{p\in\mathcal P_n}\frac1p=\sum_{a<p\le b}\frac1p.
\]
By Proposition~6 of \cite{GlasbyPraegerUnger},
\[
\frac{|T_n|}{|S_n|}
=
1-\operatorname{Prob}\bigl(m_p(g)=0\text{ for all }p\in\mathcal P_n\bigr)
\ge 1-e^{\gamma-\mu_n}.
\]
Now, for $x\ge 50$, we have
\[
\sum_{p\le x}\frac1p
=
\log\log x+B+O^*\!\left(\frac{3.69}{\log^3 x}\right),
\]
(see Table~2 and Corollary~20 of \cite{Vanlalngaia2017}). Since $a\ge 50$, this gives
\begin{align*}
\mu_n
&=
\sum_{p\le a^d}\frac1p-\sum_{p\le a}\frac1p \\
&=
\log\log(a^d)-\log\log a
+O^*\!\left(\frac{3.69}{\log^3(a^d)}\right)
+O^*\!\left(\frac{3.69}{\log^3 a}\right) \\
&=
\log d
+O^*\!\left(\frac{3.69}{(d\log a)^3}\right)
+O^*\!\left(\frac{3.69}{(\log a)^3}\right).
\end{align*}
Hence
\[
\mu_n\ge
\log d-\frac{3.69}{(d\log a)^3}-\frac{3.69}{(\log a)^3}\ge
\log d-\frac{7.38}{(\log a)^3},
\]
and therefore
\[
e^{\gamma-\mu_n}
\le
\frac{e^\gamma}{d}
\exp\!\left(
\frac{7.38}{(\log a)^3}
\right).
\]
Thus
\[
\frac{|T_n|}{|S_n|}
\ge
1-
\frac{e^\gamma}{d}
\exp\!\left(
\frac{7.38}{(\log a)^3}\right)\ge 1-\frac{2.015}{d}
.
\]
We now estimate $|U_n|$. For each prime $p\in\mathcal P_n$, the proof of
Proposition~9 of \cite{GlasbyPraegerUnger} shows that
\[
\frac{|U_p|}{|S_n|}
<
\frac{1+\log(m+1)}{p^2}+\varepsilon_p,
\qquad
m:=\left\lfloor\frac{n-2}{p}\right\rfloor-1,
\]
where $\varepsilon_p=\varepsilon$ if $n-1\le (k+1)p\le n$ has a solution in $k$,
and $\varepsilon_p=0$ otherwise, with
\[
\varepsilon=\frac{2}{a(n-1)}.
\]
Since $p>a$, we have
\[
m+1=\left\lfloor\frac{n-2}{p}\right\rfloor<\frac{n}{p},
\]
and therefore
\[
1+\log(m+1)<1+\log\!\left(\frac{n}{p}\right)
\le \log n+1-\log a.
\]
It follows that
\[
\frac{|U_p|}{|S_n|}
\le
\frac{\log n+1-\log a}{p^2}+\varepsilon_p.
\]
Hence
\[
\frac{|U_n|}{|S_n|}
\le
(\log n+1-\log a)\sum_{a<p\le b}\frac1{p^2}
+\sum_{a<p\le b}\varepsilon_p.
\]
As in \cite[Proposition~9]{GlasbyPraegerUnger},
\[
\sum_{a<p\le b}\varepsilon_p
\le 2\varepsilon\log_a n
=
\frac{4\log_a n}{a(n-1)}.
\]
Now apply Corollary~\ref{cor:prime-square-sum-sharper-real} with $A=50$ to obtain
\[
\sum_{a<p\le b}\frac1{p^2}
\le
\frac{1+\dfrac{3}{\log 50}}{\lfloor a\rfloor\log \lfloor a\rfloor}
-
\frac{1+\dfrac{3}{\log 50}-\dfrac{3}{2\log \lfloor b\rfloor}}
{\lfloor b\rfloor\log \lfloor b\rfloor}.
\]Therefore
\[
\frac{|U_n|}{|S_n|}
\le
(\log n+1-\log a)
\left(
\frac{1+\dfrac{3}{\log 50}}{\lfloor a\rfloor\log \lfloor a\rfloor}
-
\frac{1+\dfrac{3}{\log 50}-\dfrac{3}{2\log \lfloor b\rfloor}}
{\lfloor b\rfloor\log \lfloor b\rfloor}
\right)
+\frac{4\log_a n}{a(n-1)},
\]
and dropping all the terms with $b$ we have
\[
\frac{|U_n|}{|S_n|}
\le
(\log n+1-\log a)
\left(
\frac{1+\dfrac{3}{\log 50}}{\lfloor a\rfloor\log \lfloor a\rfloor}
\right)
+\frac{4\log n}{a\log a(n-1)}.
\]
Combining the estimates for $|T_n|$ and $|U_n|$ gives
\[
\pi_n\ge \frac{|T_n|-|U_n|}{|S_n|}\ge 1-\frac{2.015}{d}-(\log n+1-\log a)\frac{1+\dfrac{3}{\log 50}}{\lfloor a\rfloor\log \lfloor a\rfloor}-\frac{4.09\log n}{an\log a},
\]
since $4/(n-1)<4.09/n$ for $n\ge a(n)\ge  50$.
The required inequality follows.
\end{proof}

\section{Proof of Theorem \ref{th1}}\label{jar}
We divide into four ranges depending on the size of $n$.\subsection{Case 1: $8\le n\le 49$.}
For these values of $n$, we compute $\pi_n$ exactly. Let
\[
\lambda=\{1^{m_1}2^{m_2}\cdots n^{m_n}\}\vdash n
\]
be a partition of $n$, and let
\[
C(\lambda):=\prod_{k=1}^n k^{m_k}m_k!.
\]
The proportion of elements of $S_n$ having cycle type $\lambda$ is $C(\lambda)^{-1}$.
Moreover, by \cite[Corollary~4]{GlasbyPraegerUnger}, a permutation of cycle type $\lambda$ is a pre-$p$-cycle,
for a prime $2\le p\le n-3$, if and only if
\begin{equation}\label{eq:mip}
    m_p=1
\qquad\text{and}\qquad
m_{ip}=0 \quad \text{for all } i\ge 2.
\end{equation}
Therefore
\[
\pi_n
=
\sum_{\lambda\vdash n}
\frac{1}{C(\lambda)}
\mathbf{1}\!\left(
\exists\text{ prime }2\le p\le n-3
\, |\, 
m_p=1,\, 
m_{ip}=0\, \forall\, i\ge 2
\right),
\]
where $\mathbf{1}(\cdots)$ denotes the indicator function.
We evaluated this finite sum exactly for each
$8\le n\le 49$, using a Python program. The computation shows that
\[
\pi_n>\frac{1}{3},
\qquad\text{for all }8\le n\le 49.
\]
In particular, the minimum is reached at $n=8$, where $\pi_8=553/1440=0.384027\dots$.
\subsection{Case 2: $50\le n\le 9052$.}

For each prime $p$ with $n/40<p\le n-3$, let $A_p$ denote the event that
a permutation in $S_n$ is a pre-$p$-cycle. Since every element counted by
one of the events $A_p$ is counted by $\pi_n$, we have
\begin{equation}\label{eq:pinconnectionprob}
    \pi_n\ge \Pr\!\left(\bigcup_{n/40<p\le n-3}A_p\right).
\end{equation}
We estimate the probability on the right-hand side using the second moment
method. Set
$
\mathcal Q_n:=\{p\text{ prime}: n/40<p\le n-3\}
$
and define the variable
\[
X_n:=\sum_{p\in\mathcal Q_n}\mathbf 1_{A_p}.
\]
Then the event $X_n>0$ is precisely the event that at least one of the
events $A_p$ occurs. Hence \eqref{eq:pinconnectionprob} gives
$
\pi_n\ge \Pr(X_n>0).
$
By the second moment method,
\begin{equation}\label{eq:secondmoment}
    \Pr(X_n>0)\ge \frac{\mathbb E(X_n)^2}{\mathbb E(X_n^2)}.
\end{equation}
We first treat the range $200\le n\le 9000$. Define
\[
M_n:=
\sum_{\substack{p\text{ prime}\\ n/40<p\le n-3}}
\frac1p
\left(
1-\frac{H_{\lfloor(n-p)/p\rfloor}}{p}
\right),
\]
where $H_m$ denotes the $m$-th harmonic number, and 
\[
R_n:=
\sum_{\substack{p<q\\ n/40<p<q\le n-3\\ p+q\le n}}
\frac1{pq}.
\]
We now estimate $\mathbb E(X_n)$. Recall from \eqref{eq:mip} that a
permutation is a pre-$p$-cycle if and only if
\[
m_p=1,
\qquad
m_{ip}=0 \quad (i\ge 2).
\]
Thus
\[
\mathbb E(X_n)=\sum_{p\in\mathcal Q_n}\Pr(A_p).
\]
Fix $p\in\mathcal Q_n$, and put
\[
r_p:=\left\lfloor\frac{n-p}{p}\right\rfloor.
\]
If a permutation lies in $A_p$, then it has exactly one $p$-cycle. After
removing this $p$-cycle, the remaining permutation on $n-p$ points has no
cycles of lengths
\[
p,2p,\ldots,r_pp.
\]
The chosen $p$-cycle contributes a factor $1/p$, and so
\[
\Pr(A_p)
=
\frac1p
\Pr_{S_{n-p}}\left(
\text{no cycle has length }p,2p,\ldots,r_pp
\right).
\]
The probability that a permutation on $n-p$ points has a cycle of length
$kp$ is at most $1/(kp)$. Hence, by the union bound,
\[
\Pr_{S_{n-p}}\left(
\text{some cycle has length }p,2p,\ldots,r_pp
\right)
\le
\sum_{k=1}^{r_p}\frac1{kp}
=
\frac{H_{r_p}}p.
\]
Therefore
\[
\Pr_{S_{n-p}}\left(
\text{no cycle has length }p,2p,\ldots,r_pp
\right)
\ge
1-\frac{H_{r_p}}p,
\]
and hence
\[
\Pr(A_p)
\ge
\frac1p\left(1-\frac{H_{r_p}}p\right).
\]
It follows that
\[
\mathbb E(X_n)
\ge
\sum_{p\in\mathcal Q_n}
\frac1p
\left(
1-\frac{H_{\lfloor(n-p)/p\rfloor}}p
\right)
=
M_n.
\]
We next estimate $\mathbb E(X_n^2)$. Since
\[
X_n^2
=
\sum_{p\in\mathcal Q_n}\mathbf 1_{A_p}
+
2\sum_{\substack{p<q\\ p,q\in\mathcal Q_n}}
\mathbf 1_{A_p}\mathbf 1_{A_q},
\]
we have
\[
\mathbb E(X_n^2)
=
\mathbb E(X_n)
+
2\sum_{\substack{p<q\\ p,q\in\mathcal Q_n}}
\Pr(A_p\cap A_q).
\]
Fix distinct primes $p<q$ in $\mathcal Q_n$. If $p+q>n$, then
$A_p\cap A_q=\varnothing$, since a permutation cannot contain both a
$p$-cycle and a $q$-cycle. If $p+q\le n$, then
\[
\mathbf 1_{A_p\cap A_q}\le m_p(g)m_q(g).
\]
Therefore
\[
\Pr(A_p\cap A_q)
\le
\mathbb E\bigl(m_p(g)m_q(g)\bigr)
=
\frac1{pq}.
\]
The equality follows from a standard cycle-counting argument. Indeed, $m_p(g)m_q(g)$ counts the number of ordered choices of a $p$-cycle and a $q$-cycle in $g$. To compute $E\bigl(m_p(g)m_q(g)\bigr)$, we count triples $(g,C_p,C_q)$, where $g\in S_n$, $C_p$ is a $p$-cycle of $g$, and $C_q$ is a $q$-cycle of $g$. There are
$
\binom{n}{p}
$
ways to choose the $p$ points forming the $p$-cycle. After this, there are
$
\binom{n-p}{q}
$
ways to choose the $q$ points forming the $q$-cycle. On these chosen points,
there are $(p-1)!$ possible $p$-cycles and $(q-1)!$ possible $q$-cycles.
The remaining $n-p-q$ points may be permuted arbitrarily, giving
$(n-p-q)!$ choices. Hence the number of such triples is
\[
\binom{n}{p}\binom{n-p}{q}(p-1)!(q-1)!(n-p-q)!
=
\frac{n!}{pq}.
\]
Dividing by $|S_n|=n!$, we obtain
$
\mathbb E\bigl(m_p(g)m_q(g)\bigr)=(pq)^{-1}.
$
Thus
$
\Pr(A_p\cap A_q)\le (pq)^{-1}.
$
Consequently,
\[
\sum_{\substack{p<q\\ p,q\in\mathcal Q_n}}
\Pr(A_p\cap A_q)
\le
\sum_{\substack{p<q\\ n/40<p<q\le n-3\\ p+q\le n}}
\frac1{pq}
=
R_n.
\]
Therefore
$
\mathbb E(X_n^2)\le \mathbb E(X_n)+2R_n.
$
Substituting this into \eqref{eq:secondmoment}, we get
\[
\Pr(X_n>0)
\ge
\frac{\mathbb E(X_n)^2}{\mathbb E(X_n)+2R_n}.
\]
For fixed $R_n\ge 0$, the function
\[
f(x):=\frac{x^2}{x+2R_n}
\]
is increasing for $x\ge 0$, since
\[
f'(x)=
\frac{x(x+4R_n)}{(x+2R_n)^2}\ge 0.
\]
Using $\mathbb E(X_n)\ge M_n$, we therefore obtain
\[
\Pr(X_n>0)
\ge
\frac{M_n^2}{M_n+2R_n}.
\]
Hence
\[
\pi_n
\ge
\Pr(X_n>0)
\ge
\frac{M_n^2}{M_n+2R_n}.
\]
A direct computation gives
\[
\frac{M_n^2}{M_n+2R_n}>\frac13
\qquad
(200\le n\le 9052).
\]
The minimum is approximately $0.34574$, attained at $n=8010$.\\
It remains to treat $50\le n\le 199$. In this range, the preceding lower
bound for $\Pr(A_p)$ is not strong enough, so we compute the two moments in
\eqref{eq:secondmoment} exactly.\\
First fix $p\in\mathcal Q_n$. The event $A_p$ means that the permutation has
exactly one $p$-cycle and no other cycle whose length is divisible by $p$.
After choosing the $p$-cycle, the remaining permutation is on $n-p$ points
and must have no cycles of lengths $p,2p,3p,\ldots$.
Recall that the exponential generating function for permutations in which
cycles of lengths belonging to a set $S$ are forbidden is
\begin{equation}\label{eq:genfunctforbidden}
\exp\left(\sum_{\substack{k\ge 1\\ k\notin S}}\frac{x^k}{k}\right)
=
\frac{1}{1-x}
\exp\left(-\sum_{k\in S}\frac{x^k}{k}\right).
\end{equation}
Taking $S$ to be the set of positive multiples of $p$, we obtain
\[
\Pr(A_p)
=
\frac1p
[x^{n-p}]
\frac{(1-x^p)^{1/p}}{1-x}.
\]
Here $[x^r]F(x)$ denotes the coefficient of $x^r$ in the power series $F(x)$.
Therefore
\[
\mathbb E(X_n)
=
\sum_{p\in\mathcal Q_n}
\frac1p
[x^{n-p}]
\frac{(1-x^p)^{1/p}}{1-x}.
\]

Next let $p<q$ be primes in $\mathcal Q_n$. If $p+q>n$, then $A_p\cap A_q=\varnothing$. If $p+q\le n$, then after choosing one $p$-cycle and one $q$-cycle, the remaining permutation is on $n-p-q$ points. In order for the original permutation to lie in $A_p\cap A_q$, the remaining permutation must have no cycle whose length is divisible by $p$ or by $q$. Using \eqref{eq:genfunctforbidden}, the forbidden set is
\[
S=\{kp:k\ge 1\}\cup\{kq:k\ge 1\}.
\]
Thus
\[
\sum_{k\in S}\frac{x^k}{k}
=
\sum_{i\ge 1}\frac{x^{ip}}{ip}
+
\sum_{j\ge 1}\frac{x^{jq}}{jq}
-
\sum_{\ell\ge 1}\frac{x^{\ell pq}}{\ell pq}.
\]
The final term is subtracted because multiples of $pq$ have been counted
twice. Hence
\[
\exp\left(-\sum_{k\in S}\frac{x^k}{k}\right)
=
(1-x^p)^{1/p}
(1-x^q)^{1/q}
(1-x^{pq})^{-1/(pq)}.
\]
Therefore
\[
\Pr(A_p\cap A_q)
=
\frac1{pq}
[x^{n-p-q}]
\frac{(1-x^p)^{1/p}(1-x^q)^{1/q}(1-x^{pq})^{-1/(pq)}}{1-x}.
\]
Define
\[
M_n^\ast:=
\sum_{p\in\mathcal Q_n}
\frac1p
[x^{n-p}]
\frac{(1-x^p)^{1/p}}{1-x}
\]
and
\[
R_n^\ast:=
\sum_{\substack{p<q\\ p,q\in\mathcal Q_n\\ p+q\le n}}
\frac1{pq}
[x^{n-p-q}]
\frac{(1-x^p)^{1/p}(1-x^q)^{1/q}(1-x^{pq})^{-1/(pq)}}{1-x}.
\]
Then
\[
\mathbb E(X_n)=M_n^\ast,
\qquad
\mathbb E(X_n^2)=M_n^\ast+2R_n^\ast.
\]
Thus
\[
\pi_n
\ge
\Pr(X_n>0)
\ge
\frac{(M_n^\ast)^2}{M_n^\ast+2R_n^\ast}.
\]
A direct computation of these coefficients gives
\[
\frac{(M_n^\ast)^2}{M_n^\ast+2R_n^\ast}
>
\frac13
\qquad
(50\le n\le 199).
\]
The minimum is approximately $0.447389$, attained at $n=60$. Therefore
\[
\pi_n>\frac13
\qquad
(50\le n\le 9052).
\]
\subsection{Case 3: $9053\le n\le 1199999$.}
For integers $a,b$ with $12\le a<b\le n-3$, let
\[
P_n(a,b):=\{p\text{ prime}:a<p\le b\}.
\]
For each prime $p\in P_n(a,b)$, let $U_p\subseteq S_n$ be the set of permutations
having at least one $p$-cycle and at least one further cycle whose length is divisible by $p$,
and set
\[
U(n,a,b):=\frac{1}{|S_n|}\left|\bigcup_{p\in P_n(a,b)}U_p\right|.
\]
Following the same argument as in the proof of \cite[Proposition~9]{GlasbyPraegerUnger}, we overcount elements of $U_p$ by
products $g=g_1g_2g_3$ with disjoint supports, where $g_1$ is a $p$-cycle and $g_2$
is a $kp$-cycle for some $k\ge 1$. In the generic case $(k+1)p\le n-2$, the counting
argument gives the contribution
\[
\frac{1}{2p^2}+\sum_{k=2}^{\lfloor (n-2)/p\rfloor-1}\frac{1}{kp^2}
=
\frac{H_{\lfloor (n-2)/p\rfloor-1}-\frac12}{p^2},
\]
where $H_m$ denotes the $m$-th harmonic number.
In the exceptional case $n-1\le (k+1)p\le n$, there is at most one value of $k$ for each
$p$, namely $k=n/p-1$ if $p\mid n$ or $k=(n-1)/p-1$ if $p\mid (n-1)$. Therefore
\begin{align*}
U(n,a,b)\le &
\sum_{a<p \leq \min \{b,\lfloor(n-2) / 2\rfloor\}}
\frac{H_{\lfloor (n-2)/p\rfloor-1}-\frac12}{p^2}
+
\sum_{\substack{a<p\le b\\ p\mid n}}\frac{1}{(n/p-1)p^2}
\\ &
+
\sum_{\substack{a<p\le b\\ p\mid (n-1)}}\frac{1}{((n-1)/p-1)p^2}.
\end{align*}
For these values of $n$, we use a rigorous numerical implementation of \cite[Proposition 9]{GlasbyPraegerUnger}, instead of using lower bounds for finite sums of reciprocal of primes. More precisely, for integers $a,b$ with
$ 12\le a<b\le n-3,$
let
\[
\mu(a,b):=\sum_{a<p\le b}\frac1p,
\]
and define
\[
L(n;a,b):=1-e^{\gamma-\mu(a,b)}-U(n;a,b),
\]
where $U(n;a,b)$ satisfies the upper bound above.
Then
$
\pi_n\ge L(n;a,b)
$
for every admissible pair $(a,b)$. A direct computation, over a finite search set of admissible pairs $(a,b)$,
shows that for each $9053\le n\le 1199999$ there exists a pair
$(a,b)$ in the search set such that
$
L(n;a,b)>\frac13.
$
Therefore
\[
\pi_n>\frac13
\qquad
(9053\le n\le 1199999).
\]
In particular, the minimum is attained at $n=9061 $ and is greater than \begin{equation*}0.333777.
\end{equation*}
\subsection{Case 4: $n\ge 1200000$.}

We apply Proposition~\ref{prop:remark11-large-range} with
\[
a(n)=(n-3)^{2/7},
\qquad
d(n)=\frac72.
\]
Then
$
a(n)^{d(n)}=n-3.
$
Moreover, for $n\ge 1200000$, we have
\[
a(n)=(n-3)^{2/7}>54>50.
\]
Thus the hypotheses of Proposition~\ref{prop:remark11-large-range} are satisfied. Hence,
\begin{equation}\label{fridge}
\pi_n\ge 
1-\frac{2.015}{7/2}
-(\log n+1-\log a(n))
\frac{1+\dfrac{3}{\log 50}}{\lfloor a(n)\rfloor\log \lfloor a(n)\rfloor}
-\frac{4.09\log n}{a(n)n\log a(n)}.
\end{equation}
It remains to verify that the right-side of (\ref{fridge}) exceeds $1/3$ for all $n\ge 1200000$. Let
\[
m=\lfloor a(n)\rfloor=\left\lfloor (n-3)^{2/7}\right\rfloor.
\]
Then $m\le a(n)<m+1$, and hence
\[
m^{7/2}+3\le n <(m+1)^{7/2}+3.
\]
Therefore
\[
\log n\le \log\!\left((m+1)^{7/2}+3\right),
\qquad
\log a(n)\ge \log m.
\]
It follows that
\[
\log n+1-\log a(n)
\le
\log\!\left((m+1)^{7/2}+3\right)+1-\log m.
\]
Also,
\[
\frac{\log n}{a(n)n\log a(n)}
\le
\frac{
\log\!\left((m+1)^{7/2}+3\right)
}
{
m\left(m^{7/2}+3\right)\log m
}.
\]
Consequently,
$
\pi_n\ge \widetilde R(m),
$
where
\begin{align*}
\widetilde R(m):=
1- & \frac{2.015}{7/2}
-
\left(1+\frac{3}{\log 50}\right)\,
\frac{
\log\!\left((m+1)^{7/2}+3\right)+1-\log m
}
{m\log m}
\\ &-
\frac{
4.09\log\!\left((m+1)^{7/2}+3\right)
}
{
m\left(m^{7/2}+3\right)\log m
}.
\end{align*}
A direct calculation shows that this function is increasing in $m$ for every $m\ge 54$, so the minimum is reached at $m=54$. Hence
\begin{equation}\label{worst}
    \pi_n\ge \Tilde{R}(54)=0.333757>\frac{1}{3}.
\end{equation}

The overall minimum occurs in Case 4 in (\ref{worst}), from which we deduce that
$$\pi_{n} > 0.333757 > 1/3, $$
for all $n\geq 8$, which proves Theorem \ref{th1}.

\section*{Acknowledgements}
We are grateful to Stephen Glasby and Cheryl Praeger for enlightening discussions on this topic.

\end{document}